\def\draft{n}
\documentclass[12pt]{amsart}
\usepackage[headings]{fullpage}
\usepackage{amssymb,amsmath,bbm}
\usepackage{epic,eepic,epsfig}
\usepackage[all]{xy}
\usepackage{url}
\usepackage[bookmarks=true,%
    colorlinks=true,%
    linkcolor=blue,%
    citecolor=blue,%
    filecolor=blue,%
    menucolor=blue,%
    urlcolor=blue,%
    breaklinks=true]{hyperref}

\newtheorem{theorem}{Theorem}[section]
\theoremstyle{definition}
\newtheorem{proposition}[theorem]{Proposition}
\newtheorem{lemma}[theorem]{Lemma}
\newtheorem{definition}[theorem]{Definition}

\newtheorem{corollary}[theorem]{Corollary}

\def\printname#1{
        \if\draft y
                \smash{\makebox[0pt]{\hspace{-0.5in}
                        \raisebox{8pt}{\tt\tiny #1}}}
        \fi
}

\newcommand{\psdraw}[2]
         {\begin{array}{c} \hspace{-1.3mm}
        \raisebox{-4pt}{\epsfig{figure=draws/#1.eps,width=#2}}
        \hspace{-1.9mm}\end{array}}

\newlength{\standardunitlength}
\catcode`\@=11
\long\def\@makecaption#1#2{%
     \vskip 10pt

\setbox\@tempboxa\hbox{
       \small\sf{\bfcaptionfont #1. }\ignorespaces #2}%
     \ifdim \wd\@tempboxa >\captionwidth {%
         \rightskip=\@captionmargin\leftskip=\@captionmargin
         \unhbox\@tempboxa\par}%
       \else
         \hbox to\hsize{\hfil\box\@tempboxa\hfil}%
     \fi}
\font\bfcaptionfont=cmssbx10 scaled \magstephalf
\newdimen\@captionmargin\@captionmargin=2\parindent
\newdimen\captionwidth\captionwidth=\hsize
\catcode`\@=12

\def\lbl#1{\label{#1}\printname{#1}}

\def\BN{\mathbbm N}
\def\BZ{\mathbbm Z}
\def\BQ{\mathbbm Q}

\def\BC{\mathbbm C}
\def\BK{\mathbbm K}
\def\BW{\mathbbm W}

\def\la{\langle}
\def\ra{\rangle}

\def\SL{\mathrm{SL}}
\def\longto{\longrightarrow}
\def\ldeg{\operatorname{ldeg}}
\def\pt{\partial}
\def\ann{\mathrm{ann}}
\def\calM{\mathcal{M}}
\def\calI{\mathcal{I}}

\begin{document}
\title%
[Irreducibility of $q$-difference operators and the knot $7_4$]%
{Irreducibility of $q$-difference operators\\ and the knot $7_4$}
\author{Stavros Garoufalidis}
\address{School of Mathematics \\
         Georgia Institute of Technology \\
         Atlanta, GA 30332-0160, USA \newline
         {\tt \url{http://www.math.gatech.edu/~stavros}}}
\email{stavros@math.gatech.edu}
\author{Christoph Koutschan}
\address{Johann Radon Institute for Computational and Applied Mathematics (RICAM)\\
         Austrian Academy of Sciences \\
         Altenberger Stra\ss e 69 \\
         A-4040 Linz, Austria \newline
         {\tt \url{http://www.koutschan.de}}}
\email{christoph.koutschan@ricam.oeaw.ac.at}
\thanks{%
S.G. was supported in part by grant DMS-0805078
  of the US National Science Foundation.\bigskip\\
{\em 2010 Mathematics Subject Classification:}
  Primary 57N10. Secondary 57M25, 33F10, 39A13.\\
{\em Key words and phrases:}
  $q$-holonomic module, $q$-holonomic sequence, creative telescoping, 
  irreducibility of $q$-difference operators, factorization of
  $q$-difference operators, qHyper, Adams operations, quantum topology,
  knot theory, colored Jones polynomial, AJ conjecture, double twist
  knot, $7_4$.\bigskip\\
First published in \emph{Algebraic \& Geometric Topology} 13(6), pp. 3261--3286,
published by Mathematical Sciences Publishers.
}

\date{October 10, 2013}

\begin{abstract}
  Our goal is to compute the minimal-order recurrence of
  the colored Jones polynomial of the $7_4$ knot, as well as for the first
  four double twist knots. As a corollary, we verify the AJ
  Conjecture for the simplest knot $7_4$ with reducible non-abelian
  $\SL(2,\BC)$ character variety. To achieve our goal, we use symbolic summation
  techniques of Zeilberger's holonomic systems approach and an
  irreducibility criterion for $q$-difference operators. For the latter we use an
  improved version of the qHyper algorithm of Abramov-Paule-Petkov\v{s}ek to
  show that a given $q$-difference operator has no linear right factors.
En route, we introduce exterior power Adams operations on the ring of bivariate
polynomials and on the corresponding affine curves.
\end{abstract}

\maketitle

\tableofcontents


\section{Introduction}
\lbl{sec.intro}

\subsection{Notation}
\lbl{sub.not}

Throughout the paper the symbol~$\BK$ denotes a field of characteristic zero;
for most applications one may think of $\BK=\BQ$. We write
$\BK[X_1,\dots,X_n]$ for the ring of polynomials in the variables
$X_1,\dots,X_n$ with coefficients in~$\BK$, and similarly
$\BK[X_1^{\pm1},\dots,X_n^{\pm1}]$ for the ring of Laurent polynomials, and
$\BK(X_1,\dots,X_n)$ for the field of rational functions.  In a somewhat
sloppy way we use angle brackets, e.g., $\BK\la X_1,\dots,X_n\ra$, to refer to
the ring of polynomials in $X_1,\dots,X_n$ with some non-commutative
multiplication. This non-commutativity may occur between variables $X_i$
and~$X_j$, or between the coefficients in $\BK$ and the variables~$X_i$.  It
will be always clear from the context which commutation rules apply. Let
$p(X,Y_1,\dots,Y_n)=\sum_{k=a}^bp_k(Y_1,\dots,Y_n)X^k$, $a,b\in\BZ$, be a
nonzero Laurent polynomial with $p_a\neq0$ and $p_b\neq0$; then we define
$\deg_X(p):=b$ and $\ldeg_X(p):=a$. As usual, $\lfloor a\rfloor$
(resp. $\lceil a\rceil$) denotes the largest integer $\leq a$ (resp. smallest
integer $\geq a$).

\subsection{The colored Jones polynomial of a knot and its recurrence}
\lbl{sub.cj}

The {\em colored Jones function} $J_{K,n}(q) \in \BZ[q^{\pm 1}]$ of a
{\em knot} $K$ in 3-space for $n \in \BN$ is a powerful knot invariant
which satisfies a linear recurrence (i.e., a linear recursion relation) with coefficients that are
polynomials in $q$ and~$q^n$ \cite{GL}. The {\em non-commutative
$A$-polynomial} $A_K(q,M,L)$ of~$K$ is defined to be the (homogeneous
and content-free) such recurrence for $J_{K,n}(q)$ that has minimal order,
written in operator notation.
(By ``content-free'' we mean that the coefficients of the recurrence, 
which are polynomials in $q$ and~$M$, do not share a common non-trivial factor.)
By definition, the non-commutative $A$-polynomial of~$K$ is an element of
the {\em localized} $q$-{\em Weyl algebra}
\[
  \BW=\BK(q,M)\la L \ra/(LM-qML)
\]
where $\BK=\BQ$ and the symbols $L$ and $M$ denote operators which act on a
sequence $f_n(q)$ by
\[
  (Lf)_n(q)=f_{n+1}(q), \qquad (Mf)_n(q)=q^n f_n(q) \,.
\]
The non-commutative $A$-polynomial of a knot
allows one to compute the {\em Kashaev invariant} of a knot in 
{\em linear time},
and to confirm numerically the {\em Volume Conjecture} of Kashaev,
the {\em Generalized Volume Conjecture} of Gukov and Garoufalidis-Le,
the {\em Modularity Conjecture} of Zagier, the {\em Slope Conjecture}
of Garoufalidis and the {\em Stability Conjecture} of Garoufalidis-Le. For a 
discussion of the above conjectures and for a survey of computations, 
see~\cite{Ga6}. This explains the importance of exact formulas for the 
non-commutative $A$-polynomial of a knot.

In \cite{Ga1} (see also \cite{Ge}) the first author formulated the
{\em AJ conjecture} which relates the specialization $A_K(1,M,L)$ with
the $A$-polynomial $A_K(M,L)$ of~$K$. The latter parametrizes the
affine variety of $\SL(2,\BC)$ representations of the knot complement,
viewed from the boundary torus~\cite{CCGLS}.

So far, the AJ conjecture has been verified only for knots whose
$A$-polynomial consists of a single multiplicity-free component (aside
from the component of abelian representations) \cite{Le,LeTran}. For
the remaining knots, and especially for the hyperbolic knots, one does
not know whether the non-commutative $A$-polynomial detects
\renewcommand{\labelenumi}{(\alph{enumi})}
\begin{enumerate}
\item all non-geometric components of the $\SL(2,\BC)$ character variety,
\item their multiplicities.
\end{enumerate}

Our goal is to compute the non-commutative $A$-polynomial of the simplest
knot whose $A$-polynomial has two irreducible components of non-abelian
$\SL(2,\BC)$ representations (see Theorem~\ref{thm.cj22}), as well
as recurrences for the colored Jones polynomials of the first four double
twist knots. En route, we will introduce {\em Adams operations} on~$\BW$
which will allow us to define Adams operations of the ring $\BQ[M,L]$ of
$A$-polynomials and their non-commutative counterparts. 

\subsection{Minimal-order recurrences}
\lbl{sub.minimal}

We split the problem of determining a minimal-order recurrence for a given
sequence into two independent parts:
\renewcommand{\labelenumi}{(\alph{enumi})}
\begin{enumerate}
\item Compute a recurrence: if the sequence is defined by a
  multidimensional sum of a proper $q$-hypergeometric term (as it is the case
  for the colored Jones polynomial), numerous algorithms can produce a linear
  recurrence with polynomial coefficients; see for
  instance~\cite{PWZ}. Different algorithms in general produce different
  recurrences, which may not be of minimal order~\cite{PauleRiese97}.
\item Show that the recurrence produced in (a) has in fact minimal
  order: this can be achieved by proving that the corresponding operator is
  irreducible in~$\BW$ . Criteria for certifying the irreducibility of a
  $q$-difference operator are presented in Section~\ref{sec.irred}.
\end{enumerate}

\subsection{The non-commutative $A$-polynomial of the $7_4$ knot}
\lbl{sub.74}

To illustrate our ideas concretely, rigorously and effectively, we 
focus on the simplest knot with reducible $A$-polynomial, namely the 
$7_4$ knot in Rolfsen's notation~\cite{Rf}:
$$\psdraw{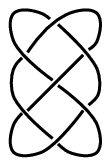}{0.8in}$$
$7_4$ is a 2-{\em bridge knot} $K(11/15)$, and a {\em double-twist} knot 
obtained by $(-1/2,-1/2)$ surgery on the {\em Borromean rings}.
Its $A$-polynomial can be computed with the {\tt Mathematica} 
implementation by Hoste or with the {\tt Maxima} implementation by
Huynh, see also Petersen~\cite{Pe}, and it is given by
\begin{multline*}
A_{7_4}(M,L) = (L^2 M^8-L M^8+L M^6+2 L M^4+L M^2-L+1)^2\\
\times (L^3 M^{14}-2 L^2 M^{14}+L M^{14}+6 L^2 M^{12}-2 L M^{12}+2 L^2 M^{10}+3 L M^{10}
-7 L^2 M^8\\
+2 L M^8+2 L^2 M^6-7 L M^6+3 L^2 M^4+2 L M^4-2 L^2 M^2+6 L M^2+L^2-2 L+1).
\end{multline*}
The first factor of $A_{7_4}(M,L)$ has multiplicity two and corresponds to a
non-geometric component of the $\SL(2,\BC)$ character variety of~$7_4$.  The
second factor of $A_{7_4}(M,L)$ has multiplicity one and corresponds to the
geometric component of the $\SL(2,\BC)$ character variety of~$7_4$.  Let
$A_{7_4}^{\text{red}}$ denote the squarefree part of the above polynomial
(i.e., where the second power of the first factor is replaced by the first
power), called the \emph{reduced $A$-polynomial}.  Finally, let
$A_{-7_4}^{\text{red}}(M,L)=A_{7_4}^{\text{red}}(M,L^{-1})L^5 \in \BZ[M,L]$
denote the reduced $A$-polynomial of~$-7_4$, the mirror of~$7_4$.

\begin{definition}
We say that an operator $P\in\BK(q)\la M^{\pm1},L^{\pm1}\ra/(LM-qML)$
is \emph{palindromic} if and only if there exist integers $a,b\in\BZ$ such that
\begin{equation}\lbl{eq.palin}
  P(q,M,L)=(-1)^aq^{bm/2}M^mL^bP(q,M^{-1},L^{-1})L^{\ell-b}
\end{equation}
where $m=\deg_M(P)+\ldeg_M(P)$ and $\ell=\deg_L(P)+\ldeg_L(P)$.  An
operator in~$\BW$ is called palindromic if, after clearing
denominators, it is palindromic in the above sense.
\end{definition}
If $P=\sum_{i,j}p_{i,j}M^iL^j$ then condition~\eqref{eq.palin} implies
that $p_{i,j} = (-1)^a q^{b(i-m/2)} p_{m-i,\ell-j}$ for all
$i,j\in\BZ$. Note also that palindromic operators give rise to (skew-)
symmetric solutions (if doubly-infinite sequences $(f_n)_{n\in\BZ}$ are
considered). More precisely, the equation $Pf=0$ for palindromic $P$
admits nontrivial symmetric (i.e., $f_{\lceil r+n\rceil}=f_{\lfloor
r-n\rfloor}$ for all~$n$) and skew-symmetric (i.e., $f_{\lceil
r+n\rceil}=-f_{\lfloor r-n\rfloor}$ for all~$n$) solutions, where
$r=(\ell-b)/2$ is the reflection point.

The next theorem gives the non-commutative $A$-polynomial of $7_4$ in
its inhomogeneous form. Every inhomogeneous recurrence $Pf=b$ gives rise
to a homogeneous recurrence $(L-1)(b^{-1}P) f =0$.

\begin{theorem}
\lbl{thm.cj22}
The inhomogeneous 
non-commutative $A$-polynomial of $7_4$ is given by the equation
\begin{equation}
\lbl{eq.74}
  P_{7_4} J_{7_4,n}(q) = b_{7_4}
\end{equation}
with $b_{7_4}\in\BQ(q,q^n)$ and $P_{7_4}\in\BW$ being a palindromic
operator of $(q,M,L)$-degree $(65,24,5)$; both are given explicitly in
Appendix~\ref{app.74}.
\end{theorem}

The proof of Theorem \ref{thm.cj22} consists of three parts:
\renewcommand{\labelenumi}{\arabic{enumi}.}
\begin{enumerate}
\item Compute the inhomogeneous recurrence~\eqref{eq.74} for the colored
Jones function $J_{7_4,n}(q)$ using the iterated double sum formula for
the colored Jones function (Equation \eqref{eq.CJpp'})
and rigorous computer algebra algorithms (see Section~\ref{sec.compute}). 
\item Prove that the operator~$P_{7_4}$ has no right factors of 
positive order (see Section~\ref{sec.irred}). To this end, we
discuss some natural $\BW$-modules associated to a knot, given
by the exterior algebra operations.
\item Show that $J_{7_4,n}(q)$ does not satisfy a zero-order
inhomogeneous recurrence, by using the degree of the colored Jones
function (see Section~\ref{sec.finish}).
\end{enumerate}

\begin{corollary}
The AJ conjecture holds for the knot~$7_4$:
\[
P_{7_4}(1,M,L) = 
A_{-7_4}^{\text{red}}(M^{1/2},L) (M-1)^5 (M+1)^4 (2M^4-5M^3+8M^2-5M+2).
\]
\end{corollary}
\begin{proof}
This follows from Theorem~\ref{thm.cj22} by setting $q=1$.
\end{proof}


\section{The colored Jones polynomial of double twist knots}
\lbl{sec.cj}

Let $J_{K,n}(q)$ denote the colored Jones polynomial of the 0-framed
knot~$K$, colored by the $n$-dimensional irreducible representation of
$\mathfrak{sl}_2(\BC)$ and normalized to be~$1$ at the
unknot~\cite{Tu1,Tu2,Ja}.  The {\em double twist knot} $K_{p,p'}$
depicted below is given by $(-1/p,-1/p')$ surgery on the Borromean
rings for integers~$p,p'$,
\[
  \psdraw{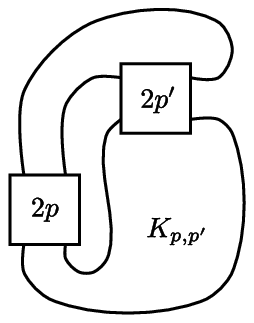}{1.3in}
\]
where the boxes indicate half-twists as follows 
\[
  \psdraw{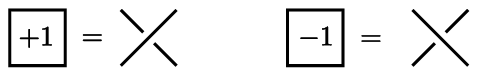}{2in}
\]
Using the Habiro theory of the colored Jones function,
(see \cite[Sec.6]{Lauridsen} following \cite{Masbaum}
and \cite{Habiro}) it follows that
\begin{equation}
\lbl{eq.CJpp'}
  J_{K_{p,p'},n}(q) = \sum_{k=0}^{n-1} (-1)^k c_{p,k}(q) c_{p',k}(q) q^{-k n -\frac{k(k+3)}{2}}
  (q^{n-1};q^{-1})_k (q^{n+1};q)_k 
\end{equation}
where $(x;q)_n$ denotes, as usual, the $q$-Pochhammer symbol defined
as $\prod_{j=0}^{n-1}(1-xq^j)$ and
\begin{equation}
\lbl{eq.cp}
  c_{p,n}(q) = \sum_{k=0}^n(-1)^{k + n} q^{-\frac{k}{2}+\frac{k^2}{2}+\frac{3 n}{2}+\frac{n^2}{2}+k p+k^2 p}
  \frac{(1 - q^{2k+1}) (q;q)_n}{(q;q)_{n-k}(q;q)_{n+k+1}} 
\end{equation}
Keep in mind that the above definition of $c_{p,n}$ differs by a power
of $q$ from the one given in~\cite[Thm.3.2]{Masbaum}. With our
definition, we have $c_{-1,n}(q)=1$ and $c_{1,n}(q)=(-1)^n q^{n(n+3)/2}$.
In \cite{GS1} it was shown that for each integer~$p$, the sequence
$c_{p,n}(q)$ satisfies a monic recurrence of order~$|p|$ with initial
conditions $c_{p,n}(q)=0$ for $n<0$ and $c_{p,0}(q)=1$. In particular,
for $p=2$ we have:
\[
  c_{2,n+2}(q) + q^{n+3}(1+q-q^{n+2}+q^{2n+4})c_{2,n+1}(q) +
  q^{2n+6}(1-q^{n+1})c_{2,n}(q) = 0 \,.
\] 

Now, $K_{2,2}=7_4$. The first few values of the colored Jones polynomial
$f_n(q):=J_{7_4,n}(q)$ are given by
\begin{alignat*}2
f_1(q) & = & \, & 1 \\
f_2(q) & = && q-2 q^2+3 q^3-2 q^4+3 q^5-2 q^6+q^7-q^8 \\
f_3(q) & = && q^2-2 q^3+q^4+4 q^5-6 q^6+2 q^7+6 q^8-9 q^9+3 q^{10}+7 q^{11}-8 q^{12}+q^{13}+7 q^{14}-{}\\
&&& 7q^{15}-q^{16}+5 q^{17}-4 q^{18}-q^{19}+3 q^{20}-q^{21}-q^{22}+q^{23} \\
f_4(q) & = && q^3-2 q^4+q^5+2 q^6-4 q^8+q^9+6 q^{10}-2 q^{11}-8 q^{12}+5 q^{13}+9 q^{14}-4 q^{15}-13 q^{16}+{}\\
&&& 7 q^{17}+11 q^{18}-3 q^{19}-15 q^{20}+6 q^{21}+11 q^{22}-q^{23}-13 q^{24}+q^{25}+10 q^{26}+2 q^{27}-{}\\
&&& 11 q^{28}-3 q^{29}+9 q^{30}+3 q^{31}-7 q^{32}-5 q^{33}+7 q^{34}+4 q^{35}-3 q^{36}-5 q^{37}+3 q^{38}+{}\\
&&& 3 q^{39}-3 q^{41}+q^{43}+q^{44}-q^{45}
\end{alignat*}
The above data agrees with the {\tt KnotAtlas}~\cite{B-N}.


\section{Computing a recurrence for the colored Jones polynomial of $7_4$}
\lbl{sec.compute}


We employ the definition of $J_{K_{p,p'},n}(q)$ given in~\eqref{eq.CJpp'}
and~\eqref{eq.cp} in terms of definite sums to compute a recurrence
for the colored Jones polynomial of~$7_4=K_{2,2}$. Thanks to
Zeilberger's \emph{holonomic systems approach}~\cite{Zeilberger90}
this task can be executed in a completely automatic fashion, e.g.,
using the algorithms implemented in the Mathematica package
\texttt{HolonomicFunctions}~\cite{Ko2}, see~\cite{Ko1} for more details.
The summation problem in~\eqref{eq.cp} can be tackled by a
$q$-analogue of Zeilberger's fast summation
algorithm~\cite{Zeilberger90a,Zeilberger91,WilfZeilberger92,PauleRiese97}
since the summand is a proper $q$-hypergeometric term.

As it was mentioned above, the sequence $c_{p,n}(q)$ satisfies a
recurrence of order~$|p|$ and therefore the summand
of~\eqref{eq.CJpp'} is not $q$-hypergeometric in general. Thus we
apply Chyzak's generalization~\cite{Chyzak00} of Zeilberger's
algorithm to derive a recurrence for~$J_{7_4,n}(q)$.

Both algorithms are based on the concept of \emph{creative
  telescoping}~\cite{Zeilberger91}, see~\cite{Ko1} for an introduction
and~\cite{GS2} for an earlier application to the computation of
non-commutative $A$-polynomials.  Let $f_{n,k}(q)$ denote the summand
of~\eqref{eq.CJpp'}. Our implementation of Chyzak's algorithm yields the equation
\[
  P_{7_4}(f_{n,k}) = c_d(q,q^n)f_{n+d,k}+\dots+c_0(q,q^n)f_{n,k} = g_{n,k+1}-g_{n,k}
\]
where $g_{n,k}$ is a $\BK(q,q^k,q^n)$-linear combination of certain shifts
of~$f$ (e.g., $f_{n,k}$, $f_{n+1,k}$, $f_{n,k+1}$, etc.).  Now creative
telescoping is executed by summing this equation with respect to~$k$.
It follows that $P_{7_4}(J_{7_4,n}) = g_{n,n}-g_{n,0} = b_{7_4}$.

The summation problems~\eqref{eq.CJpp'} and~\eqref{eq.cp} for
$p=p'=2$ are of moderate size: our
software \texttt{HolonomicFunctions} computes the solution in less
than 2 minutes. The result is given in Appendix~\ref{app.74}.


\section{Irreducibility of $q$-difference operators}
\lbl{sec.irred}

An element $P \in \BW$ is {\em irreducible} if it cannot be written in the
form $P=Q R$ with $Q,R \in \BW$ of positive $L$-degree. Since there is a (left
and right) division algorithm in~$\BW$, it follows that every element $P$ is a
finite product of irreducible elements.  However, it can happen that $P$ can
be factored in different ways, but any two factorizations of $P$ into
irreducible elements are related in a specific way; see \cite{Ore}.

A factorization algorithm for elements of the {\em localized Weyl algebra}
$\BK(x)\la \pt \ra$ where $\pt x - x \pt =1$ has been discussed by several
authors that include \cite{Schwarz,Tsarev,Bronstein} and also
\cite[Sec.8]{Hoeij1} and \cite{vPS}; the factorization of more general Ore
operators (including differential, difference, and $q$-difference) has been
investigated in~\cite{BronsteinPetkovsek96}.  Roughly, a factorization
algorithm for $P \in \BK(x)\la \pt \ra$ of order $d$ (as a linear
differential operator) proceeds as follows: if $P=QR$ where $R$ is of
order~$k$, then the coefficients of $R$ can be computed by finding the right
factors of order~$1$ of the associated equation obtained by the $k$-th
exterior power of~$P$. The problem of finding linear right factors can be
solved algorithmically.

For our purposes we do not require a full factorization algorithm, but only
criteria for certifying the irreducibility of $q$-difference operators.
Consider $P(q,M,L) \in \BW$ and assume that the leading coefficient of~$P$
does not vanish when specialized to $q=1$.  The following is an algorithm for
certifying irreducibility of~$P$:
\renewcommand{\labelenumi}{\arabic{enumi}.}
\begin{enumerate}
\item If $P(1,M,L) \in \BK(M)[L]$ is irreducible, then $P$ is irreducible (see
  Section~\ref{sub.easy}).
\item If not, factor the commutative polynomial $P(1,M,L)$ into irreducible factors $P_1\cdots P_n$.
\item For each $k=\sum_{i\in I}\deg_L(P_i)$ such that
  $I\subseteq\{1,\dots,n\}$ compute the exterior power $\bigwedge^{\!k}\!P \in
  \BW$ (see Section~\ref{sub.adams}).
\item Apply the algorithm qHyper (e.g., in its improved version described in
  Section~\ref{sec.irr74}), to show that none of the computed exterior powers
  has a linear right factor. Then $P$ is irreducible.
\end{enumerate}

\subsection{An easy sufficient criterion for irreducibility}
\lbl{sub.easy}

In this section we mention an easy irreducibility criterion in~$\BW$,
which is sufficient but not necessary, as we shall see. This criterion
has been used in~\cite{GS2} to compute the non-commutative
$A$-polynomial of twist knots, and also in~\cite{Le,LeTran} to
verify the AJ conjecture in some cases.

To formulate the criterion, we will use the Newton polygon at $q=1$,
in analogy with the Newton polygon at $q=0$ studied in~\cite{Ga3}.
Expanding a rational function $a(q,M) \in \BK(q,M)$ into a formal
Laurent series in $q-1$, let $v(a(q,M)) \in \BZ \cup\{\infty\}$ denote
the lowest power of $q-1$ which has nonzero coefficient. It can be easily
verified that $v$ is a {\em valuation}, i.e., it satisfies
\[
  v(ab)=v(a)+v(b), \qquad v(a+b) \geq \min(v(a),v(b)) \, .
\]

\begin{definition}
\lbl{def.newton}
For an operator $P(q,M,L)=\sum_{j=0}^d a_j(q,M)L^j \in \BW$ the
{\em Newton polygon}~$N(P)$ is defined to be the lower convex hull 
(see \cite{LRS}) of
the set $\{(j,v(a_j)) \mid j=0,\dots,d\}$. Furthermore, let $N^e(P)$
denote the union of the (non-vertical) boundary line segments
of~$N(P)$.
\end{definition}
For instance, if
$P=(q-1)^2 L^5 + ((q-1)(Mq-1))^{-1} L^3 + L^2 + (q-1)^{-1}L+1$,
then $N(P)$ is given by
$$\psdraw{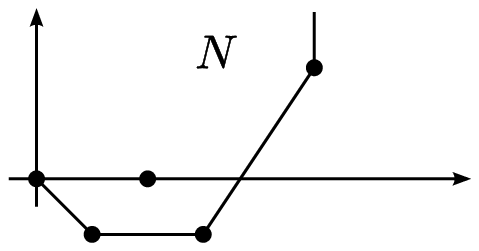}{2in}$$
and $N^e(P)$ is the path of straight line segments connecting the
points $(0,0)$, $(1,-1)$, $(3,-1)$, and~$(5,2)$.  The next lemma is
elementary; it follows easily from the definitions.
Recall \cite{Ziegler} that the {\em Minkowski sum} $A+B$ of two 
polytopes $A$ and $B$ is the convex hull of the set $\{a+b \mid a \in A, \, 
b \in B\}$.

\begin{lemma}
\lbl{lem.newton}
If $Q,R \in \BW$, then $N(QR)=N(Q)+N(R)$.
\end{lemma}

\begin{proposition}
\lbl{prop.easy}
Let $P(q,M,L)=\sum_{j=0}^d a_j(q,M)L^j \in \BW$ with $d>1$ and assume
that $P(1,M,L) \in \BK(M)[L]$ is well-defined and irreducible with
$a_0(1,M)a_d(1,M) \neq 0$. Then $P(q,M,L)$ is irreducible in~$\BW$.
\end{proposition}

\begin{proof}
The assumptions imply that $N^e(P)$ is the horizontal line segment
from the origin to $(\deg_L(P),0)$.  If $P=QR$ with
$\deg_L(Q)\deg_L(R)\neq0$, then Lemma~\ref{lem.newton} implies that
both $N^e(Q)$ and $N^e(R)$ consist of a single horizontal segment as
well.  Without loss of generality, assume that the leading coefficient
of~$Q$ has valuation zero; if not we can multiply $Q$ by $(1-q)^a$ and
$R$ by $(1-q)^{-a}$ for an appropriate integer~$a$. Then, it follows that
$N^e(Q)=[0,\deg_L(Q)]\times 0$ and $N^e(R)=[0,\deg_L(R)]\times 0$.
Evaluating at $q=1$, it follows that $Q(1,M,L),
R(1,M,L) \in \BK(M)[L]$ are well-defined and
$P(1,M,L)=Q(1,M,L)R(1,M,L)$ where $Q(1,M,L)$ and $R(1,M,L)$ are of
$L$-degree $\deg_L(Q)$ and $\deg_L(R)$ respectively. This contradicts
the assumption that $P(1,M,L)$ is irreducible and completes the proof.
\end{proof}

\subsection{Adams operations on $\BW$-modules}
\lbl{sub.adams}

In this section we introduce Adams (i.e., exterior power) operations 
on finitely generated~$\BW$-modules. The Adams operations were inspired 
by the Weyl algebra setting, and play an important role in irreducibility
and factorization of elements in~$\BW$.

To begin with, a finitely generated left $\BW$-module $\calM$ is a direct sum 
of a free module of finite rank and a cyclic torsion module. The proof of this
statement for $\BW$ is identical to the proof for modules over the Weyl
algebra, discussed for example in Lemma 2.5 and Proposition 2.9 of
\cite{vPS}.

Consider a torsion $\BW$-module $\calM$ with generator~$f$. We will write this
by $(\calM,f)$, following the notation of \cite[Sec.2.3]{vPS}.
$f$ is often called a {\em cyclic vector} for~$\calM$. It follows that 
$\calM=\BW/(\BW P)$ where $P$ is a generator of the left annihilator ideal 
$\ann(f):=\{Q \in \BW \mid Q f=0\}$ of~$f$. 

\begin{definition}
\lbl{def.extM}
For a natural number~$k$, we define the $k$-th exterior power of $(\calM,f)$
by
\[
  \textstyle{\bigwedge^{\!k}}(\calM,f)=(\textstyle{\bigwedge^{\!k}}\calM, f \wedge Lf \wedge \dots \wedge L^{k-1}f)
\]
If $P=\ann(f)$, then we define 
$\bigwedge^{\!k}\!P:=\ann(f \wedge Lf \wedge \dots \wedge L^{k-1}f)$.
\end{definition}
The next lemma is an effective algorithm to compute $\bigwedge^{\!k}\!P$.
Recall the {\em shifted analogue of the Wronskian}
(also called {\em Casoratian}) of~$k$ sequences 
$f^{(i)}_{n}$ for $i=1,\dots,k$ given by
\begin{equation}\lbl{eq.wronskian}
  W\big(f^{(1)},\dots,f^{(k)}\big)_n =
  \det_{\genfrac{}{}{0pt}{}{0 \leq j \leq k-1}{1 \leq i \leq k}} f^{(i)}_{n+j}.
\end{equation}

\begin{lemma}
\lbl{lem.1}
Let $P\in\BW$ and $f^{(1)}_n, \dots, f^{(k)}_n$ be $k$ linearly
independent solutions of the equation $Pf=0$. Then $\bigwedge^{\!k}\!P$ is
the minimal-order operator in $\BW$ which annihilates the sequence
$w_n=W\big(f^{(1)},\dots,f^{(k)}\big)_n$. In particular, there is a
unique such solution (up to left multiplication by elements from $\BK(q,M)$).
\end{lemma}

\begin{proof}
Let $d=\deg_L(P)$ and fix $k \leq d$. Let 
$\calI=\{(i_1,\dots,i_k) \mid 1 \leq i_1 < i_2 < \dots < i_k \leq d \}$.
$\bigwedge^{\!k}\!\calM$ has a basis $\{ v_I \mid I \in \calI\}$ where
$v_I=L^{i_1-1}f \wedge L^{i_2-1}f \dots \wedge L^{i_k-1}f$
for $I=(i_1,\dots,i_k) \in \calI$. Now, using the fact that $Pf=0$, it
follows that for $v\in\calM$ and every natural number~$j$ we have 
$$
L^jv=\sum_{I \in \calI} a_{j,I} v_I
$$
where the $a_{j,I}$ are 
rational functions in $q$ and~$q^n$.
It follows that the set $\{L^j v \mid j=0,\dots,\binom{d}{k}\}$ is linearly
dependent.  A minimal dependency of this set gives rise to the operator
$\bigwedge^{\!k}\!P$ by definition of the latter.

To prove the lemma, choose a fundamental set of solutions
$f^{(i)}$ for $i=1,\dots,d$ to the recurrence equation $Pf=0$ and consider
the correspondence
$$
\phi: \{ v_I \mid I \in \calI\} \longto \{ w_I \mid I \in \calI \},
\qquad \phi(v_I)=w_I
$$
where 
$w_I=W\big(f^{(i_1)},\dots, f^{(i_k)}\big)$ for 
$I=(i_1,\dots,i_k) \in \calI$. The above correspondence is invariant with
respect to the $L$-action since
\[
  L W\big(g^{(i_1)},\dots,g^{(i_k)}\big) = W\big(L g^{(i_1)},\dots,L g^{(i_k)}\big).
\]
It follows that for all natural numbers~$j$ we have
\[
  L^j w=\sum_{I \in \calI} a_{j,I} w_I
\]
where $w=\phi(f\wedge Lf \wedge\dots\wedge L^{k-1}f)=W\big(f^{(1)},\dots,f^{(k)}\big)$.
Since $\bigwedge^{\!k}\!P$ is the operator that encodes the minimal dependency
among the translates of~$v$, it follows that it is also the operator that
encodes the minimal dependency among the translates of~$w$.
The result follows.
\end{proof}

\begin{corollary}
\lbl{cor.alg.lem1}
Lemma~\ref{lem.1} gives the following algorithm to 
compute $\bigwedge^{\!k}\!P$: The
definition of~$w_n$ as a determinant together with the equations
$Pf^{(i)}=0$ for $i=1,\dots,k$ allows to express $w_{n+\ell}$ for
arbitrary $\ell\in\BN$ as a $\BK(q,q^n)$-linear combination of the
products $\prod_{i=1}^k f^{(i)}_{n+j_i}$ where $0\leq j_i<\deg_L(P)$
for $1\leq i\leq k$. This allows to determine the minimal~$\ell$ such
that $w_{n},\dots,w_{n+\ell}$ are $\BK(q,q^n)$-linearly dependent.
Compare also with \cite[Example 2.29]{vPS}. 
\end{corollary}

\begin{lemma}
\lbl{lem.2}
Let $P\in\BW$ be of the form $P=L^d+\sum_{j=0}^{d-1}a_jL^j$ with
$a_0\neq0$, and let $\big\{f^{(1)}_n,\dots,f^{(d)}_n\big\}$ be a
fundamental solution set of the equation $Pf=0$. Then
$w_{n+1}-(-1)^da_0w_n=0$ where $w=W(f^{(1)},\dots,f^{(d)})$.
\end{lemma}
\begin{proof}
The proof is done by an elementary calculation
\[
  w_{n+1} = \det\begin{pmatrix}
    f^{(1)}_{n+1} & \cdots & f^{(d)}_{n+1}\\
    \vdots & & \vdots\\
    f^{(1)}_{n+d} & \cdots & f^{(d)}_{n+d} \end{pmatrix} =
  \det\begin{pmatrix}
    f^{(1)}_{n+1} & \cdots & f^{(d)}_{n+1}\\
    \vdots & & \vdots\\
    f^{(1)}_{n+d-1} & \cdots & f^{(d)}_{n+d-1}\\
    -a_0f^{(1)}_{n} & \cdots & -a_0f^{(d)}_{n} \end{pmatrix} = (-1)^d a_0 w_n
\]
where in the second step the identities 
$f^{(i)}_{n+d}=-\sum_{j=0}^{d-1}a_jf^{(i)}_{n+j}$
and some row operations have been employed.
\end{proof}

\begin{theorem}
\lbl{thm.factor}
Let $P,Q,R\in\BW$ such that $P=QR$ is a factorization of~$P$, and let
$k$ denote the order of~$R$, i.e., $k=\deg_L(R)$. Then $\bigwedge^{\!k}\!P$
has a linear right factor of the form $L-a$ for some $a\in\BK(q,M)$.
\end{theorem}

\begin{proof}
Let $F=\big\{f^{(1)},\dots,f^{(k)}\big\}$ be a fundamental solution
set of~$R$.  By Lemma~\ref{lem.2} it follows that
$w=W(f^{(1)},\dots,f^{(k)})$ satisfies a recurrence of order~$1$, say
$w_{n+1}=aw_n,a\in\BK(q,M)$.  But $F$ is also a set of linearly
independent solutions of~$Pf=0$, and therefore $w$ is contained in the
solution space of $\bigwedge^{\!k}\!P$. It follows that
$\bigwedge^{\!k}\!P$ has the right factor~$L-a$.
\end{proof}


\section{Plethysm}
\lbl{sec.plethysm}

In this section we define Adams operations on the ring $\BQ(M)[L]$, and in
particular on the set of affine curves in $\BC^* \times \BC^*$.

Let $\BQ(M)_+[L]$ denote the subring of $\BQ(M)[L]$ which consists of 
$p(M,L) \in \BQ(M)[L]$ with constant term~$1$, i.e., $p(M,0)=1$.
If $p(M,L) \in \BQ(M)_+[L]$ has degree $d=\deg_L(p)$, then we can write
$$
p(M,L)=\prod_{i=1}^d(1+ L_i(M) L^i)
$$
in an appropriate algebraic closure of $\BQ(M)[L]$. 
\begin{definition}
\lbl{def.psiA}
For $k \in \BN$ we define
$
\psi: \BQ(M)_+[L] \longto \BQ(M)_+[L]
$ 
by
$$
\psi_k(p)(M,L)=\prod_{1 \leq i_i < i_2 < \dots < i_k \leq d}
(1+ L_{i_1}(M) \dots L_{i_k}(M) L)
$$
\end{definition}

The next lemma expresses the coefficients of $\psi_k(p)$ in terms of those of 
$p$ using the plethysm operations on the basis $e_i$ of the ring of symmetric 
functions~$\Lambda$. For a definition of the latter, see 
\cite[Sec.I.8]{Macdonald}.

\begin{lemma}
\lbl{lem.plethysm}
If $p=\prod_{i=1}^\infty(1+x_i L)=\sum_{i=0}^\infty e_i L^i$ then 
$$
\psi_k(p)=\sum_{i=0}^\infty (e_i \circ e_k) L^i
$$
\end{lemma}

\begin{corollary}
\lbl{cor.P2P3.plethysm}
In particular for $d=5$ and $k=2,3$ (as is the case of interest for the 
knot~$7_4$) the {\tt SF} package \cite{Stembridge} gives:
{\tiny
\begin{align*}
p &= 1 + e_1 L + e_2 L^2 + e_3 L^3 + e_4 L^4 + e_5 L^5
\\
\psi_2(p) &=1 
+e_2 L 
+( e_1e_3-e_4)L^2 
+(-2e_2e_4+e_3^2+e_1^2e_4-e_1e_5 )L^3 
+(e_1^3e_5+e_3e_5-e_4^2-3e_1e_2e_5+e_1e_3e_4 )L^4 \\ &
+(e_1^2e_3e_5-2e_1e_4e_5-2e_2e_3e_5+2e_5^2+e_2e_4^2 )L^5 
+(e_1e_2e_4e_5-e_1^2e_5^2+e_2e_5^2-3e_3e_4e_5+e_4^3 )L^6 \\ &
+(-e_4e_5^2+e_1e_4^2e_5-2e_1e_3e_5^2+e_2^2e_5^2 )L^7
+(e_2e_4e_5^2-e_1e_5^3 )L^8
+e_3e_5^3 L^9
+e_5^4 L^{10}
\\
\psi_3(p) &=
1 
+e_3 L 
+(e_2e_4-e_1e_5)L^2 
+(-2e_1e_3e_5-e_4e_5+e_1e_4^2+e_2^2e_5 )L^3 
+(e_1e_2e_4e_5-e_1^2e_5^2+e_2e_5^2-3e_3e_4e_5+e_4^3 )L^4 \\ &
+(-2e_2e_3e_5^2+2e_5^3+e_1^2e_3e_5^2+e_2e_4^2e_5-2e_1e_4e_5^2 )L^5 
+(-e_4^2e_5^2+e_3e_5^3+e_1^3e_5^3+e_1e_3e_4e_5^2-3e_1e_2e_5^3 )L^6 \\ &
+(e_3^2e_5^3-e_1e_5^4+e_1^2e_4e_5^3-2e_2e_4e_5^3 )L^7
+(-e_4e_5^4+e_1e_3e_5^4 )L^8
+e_2e_5^5 L^9
+e_5^6 L^{10}
\end{align*}
}
\end{corollary}


\section{Factorization of $q$-difference operators after
Bronstein-Petkov\v{s}ek}
\lbl{sec.thm.Pfactor}

This section is not needed for the results of our paper, but may be
of independent interest. In \cite{BronsteinPetkovsek96}, 
Bronstein-Petkov\v{s}ek developed a factorization algorithm for 
$q$-difference operators, and more generally, for Ore operators. A key
component of their algorithm, which predated and motivated 
the work of \cite{APP}, is
to reduce the problem of factorization into computing all
linear right factors of a finite list of so-called associated operators.  
Since this factorization algorithm is not widely known, we will describe it 
in this section, following \cite{BronsteinPetkovsek96}. All results in
this section are due to \cite{BronsteinPetkovsek96}.

\begin{definition}
\lbl{def.wl}
Let $P\in\BW$ be of the form $P=L^d+\sum_{j=0}^{d-1}a_jL^j$ with
$a_0\neq0$, and let $\big\{f^{(1)}_n,\dots,f^{(d)}_n\big\}$ be a
fundamental solution set of the equation $Pf=0$. Let
\begin{equation}
\lbl{eq.Lw}
\sum_{l=0}^d w^{(d-l)} L^l f = \det \begin{pmatrix}
f & f^{(1)} & \cdots & f^{(d)} \\
\vdots & \vdots & & \vdots \\
L^d f & L^d f^{(1)} & \cdots & L^d f^{(d)}.
\end{pmatrix}
\end{equation}
\end{definition}

\begin{lemma}
\lbl{lem.3}
With the notation of Definition \ref{def.wl} we have
\renewcommand{\labelenumi}{(\alph{enumi})}
\begin{enumerate}
\item $w^{(d-j)}/w^{(0)}=a_j$ for $j=0,\dots,d$.
\item $w^{(0)}=W(f^{(1)},\dots,f^{(d)})$ satisfies $w^{(0)}_{n+1}+(-1)^d a_0 w^{(0)}_n=0$.
\item For $j=0,\dots,d-1$ and $n \in \BN$ we have
\begin{equation}
\lbl{eq.waj}
a_j(q,q^n) w^{(d-j)}_{n+1} + (-1)^d a_j(q,q^{n+1}) a_0(q,q^n)
w^{(d-j)}_{n} =0.
\end{equation}
\end{enumerate}
\end{lemma}
\begin{proof}
Since $\sum_{l=0}^d w^{(d-l)} L^l f^{(i)}=0$ for $i=1,\dots,d$ and
$\{f^{(1)},\dots, f^{(d)}\}$ is a fundamental solution of the equation $Pf=0$,
it follows that $P=\sum_{l=0}^d w^{(d-l)} L^l f$. This proves (a).

The definition of $w^{(0)}$ implies that $w^{(0)}=W(f^{(1)},\dots,f^{(d)})$
and likewise 
$$
w^{(d)}=W(L f^{(1)},\dots,L f^{(d)})=L W(L f^{(1)},\dots,L f^{(d)}).
$$ 
Using $w^{(d)}=a_0 w^{(0)}$ (by part (a)) and the above, we obtain (b).
  
Now, (a) gives $w^{(d-j)}=a_j w^{(0)}$, hence 
$L w^{(d-j)}= (L a_j) L w^{(0)}= (-1)^{d-1} (L a_j) a_0 w^{(0)}$. Eliminating  
$w^{(0)}$, (c) follows.
\end{proof}

Lemma~\ref{lem.3} gives the following algorithm that produce a finite
set of all possible right factors $R=L^k +\sum_{j=0}^{k-1} a_j L^j$ of an 
element $P \in \BW$.

\begin{corollary}
\lbl{cor.alg.lem2}
Using the definition of $w^{(k-j)}$ for $j=0,\dots,k$ together with the 
equations $Pf^{(i)}=0$ for $i=1,\dots,k$ allows to express $w^{(k-j)}_{n+\ell}$ for
arbitrary $\ell\in\BN$ as a $\BK(q,q^n)$-linear combination of the
products $\prod_{i=1}^k f^{(i)}_{n+j_i}$ where $0\leq j_i<\deg_L(P)$
for $1\leq i\leq k$. This allows to determine the minimal~$\ell$ such
that $w^{(k-j)}_{n},\dots,w^{(k-j)}_{n+\ell}$ are $\BK(q,q^n)$-linearly dependent.
Let $\bigwedge^{\!k}_j\!P$ denote the corresponding monic minimal-order
operators.
List all right factors of $\bigwedge^{\!k}_j\!P$ using qHyper.
If $a_j=0$, include it in the list of possible values of~$a_j$. Else,
use the computed finite list and equation \eqref{eq.waj} to list all
possible values of~$a_j$. The result follows.
\end{corollary}


\section{Irreducibility of the computed recurrence for $7_4$}
\lbl{sec.irr74}

The irreducibility of a monic operator $P\in\BW$ of order~$d$ can be
established by Theorem~\ref{thm.factor}, i.e., by showing that none of the
exterior powers $\bigwedge^{\!k}\!P$ for $1\leq k<d$ has a linear right
factor. In the case of the fifth-order operator~$P_{7_4}$ we observed that its
$q=1$ specialization factors into two irreducible factors of $L$-degrees $2$
and~$3$, respectively (and hence Proposition~\ref{prop.easy} is not
applicable). We conclude that $P_{7_4}$ cannot have right factors of
$L$-degrees $1$ or~$4$. Thus it suffices to inspect its second and third
exterior powers only.

The computation of an exterior power $\bigwedge^{\!k}\!P$ is immediate from
its definition.  We start with an ansatz for a linear recurrence for the
Wronskian:
\begin{equation}\lbl{eq.ansatz}
  c_\ell(q,M)w_{n+\ell} + \dots + c_1(q,M)w_{n+1} + c_0(q,M)w_n = 0.
\end{equation}
In the next step, all occurrences of $w_{n+j}$ in~\eqref{eq.ansatz} are
replaced by the expansion of the determinant~\eqref{eq.wronskian}, e.g.,
for $k=2$ we have
\[
  w_{n+j} = f^{(1)}_{n+j} f^{(2)}_{n+j+1} - f^{(1)}_{n+j+1} f^{(2)}_{n+j}.
\]
As before let $d$ denote the $L$-degree of~$P$.  Now each $f^{(i)}_{n+j}$ with
$j\geq d$ is rewritten as a $\BQ(q,M)$-linear combination of
$f^{(i)}_n,\dots,f^{(i)}_{n+d-1}$, using the equation $Pf^{(i)}=0$. Finally,
coefficient comparison with respect to $f^{(i)}_{n+j},1\leq i\leq k,0\leq j<d$
yields a linear system for the unknown coefficients $c_0,\dots,c_\ell$. The
minimal-order recurrence for $w_n$ can be found by trying $\ell=0$, $\ell=1$,
\dots, until a solution is found. This methodology was employed to compute
$\bigwedge^{\!2}\!P_{7_4}$ and $\bigwedge^{\!3}\!P_{7_4}$ (see
Table~\ref{tab.extpows} for their sizes). 
\begin{table}
\begin{center}
\begin{tabular}{|l|p{1em}rp{1em}|p{1em}rp{1em}|p{1em}rp{1em}|p{0.5em}rp{0.5em}|}
\hline\rule{0pt}{1.1em}
& \multicolumn{3}{|c|}{$L$-degree}
& \multicolumn{3}{|c|}{$M$-degree}
& \multicolumn{3}{|c|}{$q$-degree}
& \multicolumn{3}{|c|}{\texttt{ByteCount}} \\ \hline\hline
\rule{0pt}{1.1em}$P_{7_4}$ &&
   5 &&&  24 &&&   65 &&&   463544 &\\ \hline
\rule{0pt}{1.1em}$\bigwedge^{\!2}\!P_{7_4}$ &&
  10 &&& 134 &&&  749 &&& 37293800 &\\ \hline
\rule{0pt}{1.1em}$\bigwedge^{\!3}\!P_{7_4}$ &&
  10 &&& 183 &&& 1108 &&& 62150408 &\\ \hline
\end{tabular}
\caption{Some statistics concerning $P_{7_4}$ and its exterior powers}
\lbl{tab.extpows}
\end{center}
\end{table}

Having the exterior powers of $P_{7_4}$ at hand, we can now apply
Theorem~\ref{thm.factor} to it: for establishing the irreducibility of
$P_{7_4}$ we have to show that its exterior powers do not have right factors
of order one. Note that for our application we would not necessarily need the
minimal-order recurrences for the Wronskian---as long as they have no linear
right factors, the irreducibility follows as a consequence. Note also that one
could try to use Proposition~\ref{prop.easy} for this task; unfortunately this
is not going to work, since from the discussion in Section~\ref{sec.plethysm}
it is clear that, after the substitution $q=1$, the exterior powers in
question do have a linear factor.

It is well known that a linear right factor of a $q$-difference equation
corresponds to a $q$-hypergeometric solution, i.e., a solution $f_n(q)$ such
that $f_{n+1}/f_n$ is a rational function in~$q$ and~$q^n$. The problem of
computing all such solutions has been solved in~\cite{APP} and the
corresponding algorithm has been implemented by Petkov\v{s}ek in the
Mathematica package \texttt{qHyper}.

Let $P(q,M,L)=\sum_{i=0}^d p_i(q,M)L^i$ be an operator such that all $p_i$ are
polynomials in $q$ and~$M$.  The qHyper algorithm described in~\cite{APP} attempts to find a
right factor $L-r(q,M)$ of~$P$ where the rational function~$r$ is assumed to
be written in the normal form
\[
  r(q,M) = z(q)\frac{a(q,M)}{b(q,M)}\frac{c(q,qM)}{c(q,M)}
\]
subject to the conditions
\begin{equation}\lbl{eq.gcd}
  \gcd(a(q,M),b(q,q^nM))=1 \text{ for all } n\in\BN,
\end{equation}
and
\[
  \gcd(a(q,M),c(q,M))=1,\quad
  \gcd(b(q,M),c(q,qM))=1,\quad
  c(0)\neq0
\]
(see \cite{APP} for the existence proof).  It is not difficult to show that
under these assumptions $a(q,M)\mid p_0(q,M)$ and $b(q,M)\mid
p_d(q,q^{1-d}M)$.  Therefore the algorithm qHyper proceeds by testing all
admissible choices of $a$ and~$b$. Each such choice yields a $q$-difference
equation for $c(q,M)$ which also involves the unknown algebraic
expression~$z(q)$.  The techniques for solving this kind of equations (or for
showing that no solution exists) are described in detail in~\cite{APP}.

Now let's apply qHyper to $P^{(2)}(q,M,L):=\bigwedge^{\!2}\!P_{7_4}$ whose
trailing and leading coefficients are given by
\begin{align*}
p_0(q,M) & = q^{162}M^{44}(M-1)
  \bigg(\prod_{i=6}^9(q^iM-1)\bigg)
  \bigg(\prod_{i=6}^{10}(q^iM+1)(q^{2i+1}M^2-1)\bigg) F_1(q,M)\\
p_{10}(q,q^{-9}M) & = q^{-397}(q^2M-1)
  \bigg(\prod_{i=4}^7(M-q^i)\bigg)
  \bigg(\prod_{i=4}^8(M+q^i)(M^2-q^{2i+1})\bigg) F_2(q,M)\\
\end{align*}
where $F_1$ and $F_2$ are large irreducible polynomials, related by
$q^{280}F_1(q,M)=F_2(q,q^{10}M)$. A blind application of qHyper would
result in $45\cdot 2^{16}\cdot 2^{16}=193\,273\,528\,320$ possible choices for
$a$ and~$b$---far too many to be tested in reasonable
time. In~\cite{CluzeauHoeij06,Horn08} improvements to qHyper have
been presented which are based on local types and exclude a large number of
possible choices; however, the simple criteria described below seem to be more
efficient.

In order to confine the number of qHyper's test cases we exploit two facts.
The first is the fact that $P^{(2)}(1,M,L)=R_1(M)\cdot (L-M^4)\cdot
Q_1(M,L)\cdot Q_2(M,L)$ where $Q_1$ and $Q_2$ are irreducible of $L$-degree
$3$ and~$6$, respectively. In other words, we need only to test pairs $(a,b)$
which satisfy the condition
\begin{equation}\lbl{eq.q1}
  a(1,M)=M^4b(1,M).
\end{equation}
The second fact is that $a$ and $b$ must fulfill condition~\eqref{eq.gcd}; in
Remark 4.1 of~\cite{Petkovsek92} this improvement is already suggested,
formulated in the setting of difference equations. In our example we are
lucky because the two criteria exclude most of the possible choices for $a$
and~$b$; the process of figuring out which cases remain to be tested is now
presented in detail.
\renewcommand{\labelenumi}{\arabic{enumi}.}
\begin{enumerate}
\item \eqref{eq.q1} implies that either both $F_1$ and $F_2$ must be present
  or none of them; condition~\eqref{eq.gcd} then excludes them entirely.
\item Clearly the factor $M^4$ in~\eqref{eq.q1} can only come from $M^{44}$
  in~$p_0$; thus all other (linear and quadratic) factors in $a(1,M)/b(1,M)$
  must cancel completely.
\item The most simple admissible choice is $a(q,M)=M^4$ and $b(q,M)=1$.
\item Because of~\eqref{eq.gcd} a cancellation can almost never take place
  among factors which are equivalent under the substitution $q=1$. This is
  reflected by the fact that the entries in the first column of
  Table~\ref{tab.factors} are (row-wise) larger than those in the second
  column, e.g., $(q^6M+1)\mid a(q,M)$ and $(q^{-4}M+1)\mid b(q,M)$
  violates~\eqref{eq.gcd}.
\item The only exception is that $(M-1)\mid a(q,M)$ cancels with $(q^2M-1)\mid
  b(q,M)$ in $a(1,M)/b(1,M)$. In that case, \eqref{eq.gcd} excludes further
  factors of the form $q^iM-1$, and together with~\eqref{eq.q1} we see that no
  other factors at all can occur. This gives the choice $a(q,M)=M^4(M-1)$ and
  $b(q,M)=q^2M-1$.
\item We may assume that $a(q,M)$ contains some of the quadratic factors
  $q^iM^2-1$. For $q=1$ they factor as $(M-1)(M+1)$ and therefore can be
  canceled with corresponding pairs of linear factors in
  $b(q,M)$. Condition~\eqref{eq.gcd} forces $a(q,M)$ to be free of linear
  factors and $b(q,M)$ to be free of quadratic factors. Thus we obtain
  $\sum_{m=1}^5\binom{5}{m}^3=2251$ possible choices.
\item Analogously $a(q,M)$ can have some linear factors which for $q=1$ must
  cancel with quadratic factors in $b(q,M)$; this gives $2251$ further
  choices.
\end{enumerate}
Summing up, we have to test $4504$ cases which can be done in relatively short
time on a computer. None of these cases delivered a solution for $c(q,M)$
and $z(q)$ which proves that $P^{(2)}$ does not have a linear right factor.
\begin{table}
\begin{center}
\begin{tabular}{|l|p{7em}|p{7em}||p{8em}|p{8em}|}\hline
\rule{0pt}{1.1em} & \multicolumn{2}{|c||}{$\bigwedge^{\!2}\!P_{7_4}$}
  & \multicolumn{2}{|c|}{$\bigwedge^{\!3}\!P_{7_4}$}\\ \hline
\rule{0pt}{1em} & $p_0(q,M)$ & $p_{10}(q,q^{-9}M)$
  & $p_0(q,M)$ & $p_{10}(q,q^{-9}M)$ \\ \hline\hline
\rule{0pt}{1em} $q^iM-1$ & $0$, $6$, $7$, $8$, $9$ & $-7$, $-6$, $-5$, $-4$, $2$
  & $0$, $7$, $8$, $9$ & $-6$, $-5$, $-4$, $3$ \\ \hline
\rule{0pt}{1em} $q^iM+1$
  & $6$, $7$, $8$, $9$, $10$ & $-8$, $-7$, $-6$, $-5$, $-4$
  & $7$, $8$, $9$, $10$, $11$ & $-8$, $-7$, $-6$, $-5$, $-4$\\ \hline
\rule{0pt}{1em} $q^iM^2-1$
  & $13$, $15$, $17$, $19$, $21$ & $-17$, $-15$, $-13$, $-11$, $-9$
  & $5$, $7$, $9$, $11$, $13^2$, $15^2$, $17^2$, $19^2$, $21^2$, $23$
  & $-17$, $-15^2$, $-13^2$, $-11^2$, $-9^2$, $-7^2$, $-5$, $-3$, $-1$, $1$ \\ \hline
\end{tabular}
\caption{Factors of the leading and trailing coefficients of the exterior
  powers of $P_{7_4}$; each cell contains the values of~$i$ of the
  corresponding factors. Superscripts indicate that factors occur
  with multiplicities.}\lbl{tab.factors}
\end{center}
\end{table}

The situation for $P^{(3)}(q,M,L):=\bigwedge^{\!3}\!P_{7_4}$ is very similar.
Now the trailing and leading coefficients turn out to be
\begin{align*}
p_0(q,M) & = q^{297}M^{66}(M-1)(q^7M-1)\cdots(q^{23}M^2-1)F_3(q,M)\\
p_{10}(q,q^{-9}M) & = q^{-456}(q^3M-1)(M-q^4)\cdots(M^2-q^{17})F_4(q,M)
\end{align*}
where the linear and quadratic factors can be extracted from
Table~\ref{tab.factors}. Also not explicitly displayed are the large
irreducible factors $F_3$ and $F_4$ which satisfy
$q^{275}F_3(q,M)=F_4(q,q^{10}M)$. For $q=1$ we obtain the factorization
$P^{(3)}(1,M,L)=R_2(M)\cdot(L+M^7)\cdot Q_3(M,L)\cdot Q_4(M,L)$ where $Q_3$
and $Q_4$ are irreducible of $L$-degree $3$ and~$6$, respectively.  As before
we get two special cases, the first with $a(q,M)=M^7$ and $b(q,M)=1$, and the
second with $a(q,M)=M^7(M-1)$ and $b(q,M)=q^3M-1$.  For the choices where we
cancel quadratic against linear factors, we obtain
\[
  2\sum_{m=1}^4\binom{4}{m}\binom{5}{m}
  \sum_{j=0}^{\lfloor m/2\rfloor}\binom{5}{j}\binom{10-j}{m-2j}=23600
\]
possibilities. Again, none of these cases yields a solution for $c(q,M)$
and therefore we have shown that $P^{(3)}$ does not have a linear right
factor.

Theorem~\ref{thm.factor} now implies that $P_{7_4}$ cannot have a right factor
of order $2$ or~$3$. We conclude that the operator $P_{7_4}$ is irreducible.


\section{No recurrence of order zero}
\lbl{sec.finish}

In this section we give an elementary criterion to deduce that a
$q$-holonomic sequence does not satisfy an inhomogeneous
recurrence of order zero, and apply it in the case of the $7_4$ knot to
conclude the proof of Theorem~\ref{thm.cj22}. The next lemma is
obvious.

\begin{lemma}
\lbl{lem.noCF}
If $\deg_q(f_n(q))$ is not a linear function of~$n$, then $f_n(q)$
does not satisfy $af=b$ for $a,b \in \BK(q,q^n)$.
\end{lemma}

The degree $\deg_q(J_{K,n}(q))$ of an alternating knot is well-known
and given by a quadratic polynomial in~$n$; see for
instance~\cite{Ga2} and~\cite{Le}. In the case of the alternating
knot~$7_4$, we have
\[
  \deg_q(J_{7_4,n}(q))= \frac{7}{2} n^2 - \frac{5}{2} n -1 \,.
\]
It follows that $J_{7_4,n}$ does not satisfy an inhomogeneous
recurrence of order zero.


\section{Proof of Theorem \ref{thm.cj22}}
\lbl{sec.thm.cj22}

In this section we will finish the proof of Theorem \ref{thm.cj22}.
It follows from the following lemma, of independent interest.

\begin{lemma}
\lbl{lem.cj22}
Suppose $f$ is a $q$-holonomic sequence such that
\renewcommand{\labelenumi}{(\alph{enumi})}
\begin{enumerate}
\item
$f$ satisfies the inhomogeneous recurrence $Pf=b$,
\item
$P \in \BW$ is irreducible, $\deg_L(P) > 1$ and $b \in \BK(q,q^n) \neq 0$,
\item
$f$ does not satisfy a recurrence of the form $a f = c$ for 
$a,c \in \BK(q,q^n)$, $a \neq 0$.
\end{enumerate}
Then the minimal-order homogeneous recurrence relation that $f$ satisfies
is given by $(L-1)(b^{-1} P) f = 0$. 
\end{lemma}

\begin{proof}
Let $d=\deg_L(P)$ and $P'=(L-1)(b^{-1} P)$. $P'$ is the product of two 
irreducible elements of $\BW$ (namely, $L-1$ and $b^{-1} P$) of $L$-degrees $1$ and 
$d$ respectively. Recall that $\BW$ is a Euclidean domain. 
Although the factorization of an element in $\BW$ into irreducible factors
is not unique in general,
\cite[Thm.1]{Ore} proves that the number of irreducible factors of a fixed
order is independent of the factorization. It follows
that any factorization of $P'$ into a product
of irreducible factors has exactly two factors, one of $L$-degree $1$ and 
another of $L$-degree~$d$.

Suppose $P'' f=0$ where $P''$ has minimal $L$-degree strictly less than $d+1$.
Since $\BW$ is a Euclidean domain, it follows that $P''$ is a right
factor of~$P'$, and $P''$ is a product of irreducible factors. The above 
discussion implies that $P''$ is irreducible of $L$-degree $1$ or~$d$.
Since $\BW$ is a Euclidean domain, we can write $P=Q P'' + R$ where $R \neq 0$
and $\deg_L(R) < \deg_L(P'')$. It follows that $R f =b$, thus $(L-1)(b^{-1} R)f=0$.
By the choice of~$P''$, it follows that $P''$ is a right factor of 
$(L-1)(b^{-1} R)$.\\
{\bf Case 1:} $\deg_L(P'')=1$. Then $\deg_L(R)=0$ and $f$ satisfies 
$Rf=b$ contrary to the hypothesis.\\
{\bf Case 2:} $\deg_L(P'')=d$. Then, $P''$ is irreducible and it is
a right factor of $(L-1)(b^{-1} R)$ where $\deg_L(b^{-1} R)< d$. It follows that
any factorization of $b^{-1} R$, extended to a factorization of $(L-1)(b^{-1} R)$,
will contain an irreducible factor of $L$-degree~$d$. This is impossible
since $\deg_L(b^{-1} R)< d$.
\end{proof}


\section{Extension to double twist knots}
\lbl{sec.double.twist}

\subsection{The $A$-polynomial of double twist knots}
\lbl{sub.AKpp}

The $\SL(2,\BC)$ character variety of non-abelian representations of 
$K_{p,p}$ for $p>1$ consists of two components, the geometric one, and 
the non-geometric one \cite{Pe}. It follows that the $A$-polynomial of
$K_{p,p}$ is the product of two factors, with multiplicities.
The values of $A_{K_{p,p}}(M,L)$ for $p=2,\dots,8$, as well as the recurrences
presented in Section~\ref{sub.AqKpp}, are available from
\begin{center}
{\tt \url{http://www.math.gatech.edu/~stavros/publications/double.twist.data/}}
\end{center}
For $p=2,\dots,8$ we have
\[
  A_{K_{p,p}}(M,L)=A^{\text{geom}}_{K_{p,p}}(M,L) A^{\text{ngeom}}_{K_{p,p}}(M,L)^2
\] 
is the product of two irreducible factors: the geometric component has
$(M,L)$-degree $(2p-1,8p-2)$ and multiplicity one, and the non-geometric
one has $(M,L)$-degree $(p^2-p,4p^2-4)$ and multiplicity two.  The
Newton polygons of $A^{\text{geom}}_{K_{p,p}}$ and
$A^{\text{ngeom}}_{K_{p,p}}$ are parallelograms given by the
convex hull of
$$
  \big\{(2 p - 1, 8 p - 2), (1, 8 p - 2), (0, 0), (2 p - 2, 0)\big\}
$$
and
$$
  \big\{(p^2 - p, 4 p^2 -4 p), (p - 1, 4 p^2 - 4 p), (0, 0), (p^2 - 2 p + 1, 0)\big\}
$$
respectively in $(M,L)$-coordinates. The area of the above Newton
polygons is $4(4p-1)(p-1)$ and $4p(p-1)^3$, respectively.  The
behavior of the Newton polygon of $A_{p,p}(M,L)$ as a function of~$p$
is in agreement with a theorem of~\cite{Ga5}.

\subsection{The non-commutative $A$-polynomial of double twist knots}
\lbl{sub.AqKpp}

We have rigorously computed an inhomogeneous recurrence for the double
twist knot~$K_{3,3}$, using the creative telescoping algorithm proposed in~\cite{Koutschan10c}, see
also Section~\ref{sec.compute}.  It has order~$11$ and its
$(q,q^n)$-degree is $(458,74)$. Moreover, it verifies the AJ
conjecture using the reduced $A$-polynomial. The corresponding
operator factors for $q=1$ into two irreducible factors of $L$-degrees
$5$ and~$6$. In order to show the irreducibility of the operator
itself (to prove that the computed recurrence is of minimal order), we
would have to investigate its fifth and sixth exterior powers---a
challenge that currently seems hopeless.

For $K_{4,4}$ and $K_{5,5}$ we were able to obtain recurrences, using
an ansatz with undetermined coefficients (``guessing''). Although they
were derived in a non-rigorous way, they both confirm the AJ
conjecture using the reduced $A$-polynomial. Again, both recurrences
are inhomogeneous; the one for $K_{4,4}$ has order~$19$ and
$(q,q^n)$-degree $(2045,184)$, the one for $K_{5,5}$ is a truly
gigantic one: it is of order~$29$, has $(q,q^n)$-degree $(6922,396)$,
and its total size is nearly 8GB (according to
Mathematica's \texttt{ByteCount}). These data qualify it as a good
candidate for the largest $q$-difference equation that has ever been
computed explicitly.  A rigorous derivation of these two recurrences
using creative telescoping, or even the application of the
irreducibility criterion using exterior powers, is far beyond our
current computing abilities.

\subsection*{Acknowledgment}
The authors wish to thank Marko Petkov\v{s}ek for sharing his
expertise on the qHyper algorithm and the factorization of linear
operators.


\appendix

\section{The formula for the non-commutative $A$-polynomial of $7_4$}
\lbl{app.74}

In the following, Equation~\ref{eq.74} from Theorem~\ref{thm.cj22} is
given explicitly; note that the operator $P_{7_4}(q,M,L)=\sum_{j=0}^5
a_j(q,M)L^j$ is palindromic since
$a_j(q,M)=-q^{60}M^{24}a_{5-j}(q,(q^5M)^{-1})$ (and therefore only
$a_5$, $a_4$, and $a_3$ are displayed).
{\tiny
\begin{alignat*}{2}
a_{5} & = &\,& (qM-1)(qM+1)(qM^2-1)(q^2M-1)(q^2M+1)(q^3M^2-1)(q^5M-1)(q^8(q+1)M^4-q^5(q^3+2q^2+q+1)M^3+{}\\
&&& q^2(2q^4+q^3+2q^2+2q+1)M^2-q(q^3+2q^2+q+1)M+(q+1))
\end{alignat*}
\begin{alignat*}{2}
a_{4} & = && q(qM-1)(qM+1)(qM^2-1)(q^3M^2-1)(q^4M-1)^2(q^4M+1)(q^{33}(q+1)M^{11}-q^{29}(q+2)(q^3+q+1)M^{10}+{}\\
&&& q^{24}(q+1)(2q^6-2q^5+5q^4+q^3+4q^2+3q-1)M^9-q^{20}(4q^7+2q^6+9q^5+10q^4+6q^3+6q^2-q-2)M^8-{}\\
&&& q^{16}(2q^{11}+q^9-2q^8-4q^7-12q^5-10q^4-3q^3+6q+3)M^7+{}\\
&&& q^{12}(q^{13}+2q^{12}+5q^{11}+q^{10}+4q^9-2q^7-8q^5+q^4+7q^3+7q^2+7q+2)M^6-{}\\
&&& q^9(q^{13}+3q^{12}+8q^{11}+8q^{10}+q^9+4q^8+q^7+3q^6+q^5-4q^4+7q^3+10q^2+7q+3)M^5+{}\\
&&& q^6(4q^{12}+7q^{11}+9q^{10}+4q^9-2q^8+q^7-4q^6-3q^5-3q^4-q^3+5q^2+4q+2)M^4-{}\\
&&& q^5(q^{10}+5q^9+6q^8+3q^7-7q^6-10q^5-7q^4-9q^3-9q^2-9q-3)M^3+{}\\
&&& q^2(q^2+q+1)(q^7+2q^6-5q^5-5q^4-3q^3-2q^2-3q-2)M^2+q(q^5+6q^4+9q^3+8q^2+3q+2)M-(q+1)(q+2))
\end{alignat*}
\begin{alignat*}{2}
a_{3} & = && {-q^2}(qM-1)(qM+1)(qM^2-1)(q^3M-1)^2(q^3M+1)(q^9M^2-1)(q^{41}(q+1)M^{15}-q^{37}(q^4+2q^3+3q^2+4q+1)M^{14}+{}\\
&&& q^{34}(q^5+q^4+7q^3+9q^2+8q+3)M^{13}+q^{29}(q^9+2q^8-2q^7-2q^6-10q^5-17q^4-12q^3-3q^2+2q+1)M^{12}-{}\\
&&& q^{25}(2q^{11}+4q^{10}+5q^9+4q^8-3q^7-11q^6-17q^5-11q^4+2q^3+8q^2+5q+1)M^{11}+{}\\
&&& q^{22}(6q^{11}+12q^{10}+8q^9+8q^8-14q^6-19q^5-6q^4+11q^3+16q^2+9q+2)M^{10}+{}\\
&&& q^{18}(2q^{14}-2q^{13}-9q^{12}-17q^{11}-11q^{10}+10q^8+20q^7+24q^6+7q^5-15q^4-20q^3-10q^2+1)M^9-{}\\
&&& q^{15}(q^{15}+6q^{14}-3q^{13}-14q^{12}-14q^{11}-4q^{10}+11q^9+25q^8+36q^7+35q^6+16q^5-9q^4-13q^3-6q^2+3q+3)M^8+{}\\
&&& q^{12}(4q^{15}+6q^{14}-3q^{13}-18q^{12}-16q^{11}+4q^{10}+23q^9+30q^8+39q^7+31q^6+12q^5-14q^4-14q^3-q^2+3q+3)M^7-{}\\
&&& q^9(5q^{15}+3q^{14}-11q^{13}-23q^{12}-18q^{11}+2q^{10}+19q^9+20q^8+21q^7+8q^6-7q^5-20q^4-22q^3-5q^2+q+1)M^6+{}\\
&&& q^8(q+1)(2q^{12}-4q^{11}-13q^{10}-17q^9-q^8+2q^7+11q^6-2q^5+5q^4-9q^3-13q^2-12q-6)M^5+{}\\
&&& q^5(5q^{12}+16q^{11}+25q^{10}+11q^9-8q^8-19q^7-16q^6-4q^5-2q^4+6q^3+11q^2+5q+1)M^4-{}\\
&&& q^4(2q^{10}+10q^9+9q^8-3q^7-22q^6-23q^5-20q^4-13q^3-6q^2-3q+1)M^3+{}\\
&&& q^2(q+1)(2q^7-4q^6-6q^5-17q^4-6q^3-6q^2-2q-1)M^2+q(2q^5+8q^4+11q^3+10q^2+3q+1)M-(q+1)(2q+1))
\end{alignat*}
\begin{alignat*}{2}
b_{7_4} & = && {-q^{10}}M^3(qM+1)(q^2M+1)(q^3M+1)(q^4M+1)(qM^2-1)(q^3M^2-1)(q^5M^2-1)(q^7M^2-1)(q^9M^2-1)\\
&&& (q^{10}(q^3+q^2-q+1)M^4-q^6(2q^5+2q^3+q^2-q+1)M^3+q^2(q+1)(q^7-2q^6+4q^5-q^4+q^3+q^2-q+1)M^2-{}\\
&&& q(2q^5+2q^3+q^2-q+1)M+(q^3+q^2-q+1))
\end{alignat*}
}
When $q$ is set to~$1$, the above expressions simplify drastically.
For a concise presentation we introduce the following notation for
some frequently appearing irreducible factors:
{\tiny
\begin{alignat*}{2}
v_1 & = &\,& M^4-M^3-2M^2-M+1\\
v_2 & = && M^4-2M^3+6M^2-2M+1\\
v_3 & = && 2M^4-5M^3+8M^2-5M+2\\
v_4 & = && M^7-2M^6+3M^5+2M^4-7M^3+2M^2+6M-2\\
v_5 & = && M^8-2M^7+6M^6+2M^5-10M^4+2M^3+6M^2-2M+1\\
v_6 & = && M^{12}-6M^{11}+16M^{10}-24M^9+15M^8+14M^7-36M^6+14M^5+15M^4-24M^3+16M^2-6M+1\\
v_7 & = && 2M^{14}-10M^{13}+16M^{12}-4M^{11}-46M^{10}+67M^9+28M^8-116M^7+28M^6+67M^5-46M^4-{}\\
&&& 4M^3+16M^2-10M+2\\
v_8 & = && M^{18}-4M^{17}+10M^{16}-10M^{15}-3M^{14}+40M^{13}-67M^{12}-34M^{11}+157M^{10}-14M^9-140M^8+{}\\
&&& 40M^7+66M^6-18M^5-14M^4+4M^3+4M^2-4M+1\\
v_9 & = && M^{26}-8M^{25}+42M^{24}-142M^{23}+345M^{22}-554M^{21}+521M^{20}+51M^{19}-729M^{18}+827M^{17}+{}\\
&&& 234M^{16}-843M^{15}+707M^{14}-45M^{13}+707M^{12}-843M^{11}+234M^{10}+827M^9-729M^8+51M^7+{}\\
&&& 521M^6-554M^5+345M^4-142M^3+42M^2-8M+1
\end{alignat*}
}

Now the inhomogeneous part~$b_{7_4}$ and the operator~$P_{7_4}$, together
with its second and third exterior power, evaluated at $q=1$, 
can be written in a few lines. A bar is used to denote the mirror
of a polynomial, i.e., $\overline{v} = M^{\deg(v)}v(1/M)$.
{\small
\begin{alignat*}{2}
b_{7_4}(1,M) & = &\,& {-M^3}(M-1)^5(M+1)^9 v_3\\
P_{7_4}(1,M,L) & = && (M-1)^5(M+1)^4 v_3 (L^2-v_1L+M^4)(L^3+v_4L^2+\overline{v_4}L+M^7)\\
P_{7_4}^{(2)}(1,M,L) & = && (M-1)^{10}(M+1)^{10}(M^2+1)^2 v_2 v_3^4 v_5 v_6 v_9 (L-M^4)(L^3-\overline{v_4}L^2+M^7v_4L-M^{14})\\
&&& \times(L^6+v_1v_4L^5+v_8L^4-M^4v_1v_7L^3+M^8\overline{v_8}L^2+M^{15}v_1\overline{v_4}L+M^{26})\\
P_{7_4}^{(3)}(1,M,L) & = && (M-1)^{19}(M+1)^{20}(M^2+1)^2 v_2 v_3^6 v_5 v_6 v_9 (L+M^7)(L^3+M^4v_4L^2+M^8\overline{v_4}L+M^{19})\\
&&& \times(L^6-v_1\overline{v_4}L^5+M^4\overline{v_8}L^4+M^{11}v_1v_7L^3+M^{18}v_8L^2-M^{29}v_1v_4L+M^{40})
\end{alignat*}
}

\bibliographystyle{hamsalpha}
\bibliography{biblio}
\end{document}